\documentclass{amsart}
\usepackage{stmaryrd}
\usepackage[all]{xy}

\newtheorem{thm}{Theorem}[section]
\newtheorem{proposition}[thm]{Proposition} 
\newtheorem{lemma}[thm]{Lemma}

\theoremstyle{definition} 
\newtheorem{definition}[thm]{Definition}

\theoremstyle{remark}
\newtheorem{remark}[thm]{Remark}

\newtheorem{example}[thm]{Example}

\newtheorem*{ack}{Acknowledgements}

\newcommand{\defequal}{:=}
\newcommand{\N}{\mathbb{N}}
\newcommand{\R}{{\mathbb R}}
\newcommand{\Z}{\mathbb{Z}}
\newcommand{\VB}{\mathcal{VB}}
\newcommand{\cE}{\mathcal{E}}
\newcommand{\cM}{\mathcal{M}}
\newcommand{\cB}{\mathcal{B}}
\newcommand{\cF}{\mathcal{F}}
\newcommand{\boundary}{\partial}
\newcommand{\iso}{\cong}
\newcommand{\totaldiff}{\mathcal{D}}
\newcommand{\id}{\mathrm{id}}
\newcommand{\trans}[2]{\mathcal{T}_{#1,#2}}
\newcommand{\ddt}{\frac{\partial}{\partial t}}
\newcommand{\dee}[1]{\mathrm{d} #1 \,}
\newcommand{\deed}[2]{\frac{\dee{#1}}{\dee{#2}}}
\newcommand{\cs}{\mathrm{cs}}
\newcommand{\ch}{\mathrm{ch}}
\newcommand{\bigO}{\mathcal{O}}
\newcommand{\adjoint}{\dagger}
\newcommand{\innerprod}[1]{\langle #1 \rangle}
\newcommand{\dbracket}[1]{\llbracket #1 \rrbracket}
\renewcommand{\top}{\mathrm{top}}
\DeclareMathOperator{\rank}{rk}
\DeclareMathOperator{\Hom}{Hom}
\DeclareMathOperator{\End}{End}
\DeclareMathOperator{\str}{str}
\DeclareMathOperator{\Ber}{Ber}

\title{Lie algebroid modules and representations up to homotopy}

\author{Rajan Amit Mehta}
\address{Department of Mathematics and Statistics\\
Smith College\\
44 College Lane\\
Northampton, MA 01063}
\email{rmehta@smith.edu}

\begin{document}

\begin{abstract}
We establish a relationship between two different generalizations of Lie algebroid representations: representation up to homotopy and Vaintrob's Lie algebroid modules. Specifically, we show that there is a noncanonical way to obtain a representation up to homotopy from a given Lie algebroid module, and that any two representations up to homotopy obtained in this way are equivalent in a natural sense. We therefore obtain a one-to-one correspondence, up to equivalence.
\end{abstract}


\maketitle

\section{Introduction}

A significant problem with the usual notion of Lie algebroid representation is the lack of a well-defined adjoint representation.  The effort to resolve this problem has led to a number of proposed generalizations of the notion of Lie algebroid representation \cite{vaintrob, elw, cra:vanest, cra-fer:secondary, gra-meh:vbalg, aba-cra:rephom}, with the most popular being that of \emph{representation up to homotopy} \cite{aba-cra:rephom}.

A representation up to homotopy of a Lie algebroid $A \to M$ is a chain complex of vector bundles $(\cE, \boundary)$ over $M$ equipped with an $A$-connection $\nabla$ and maps $\omega_i : \bigwedge^i \Gamma(A) \to \End_{1-i} (\cE)$ for $i \geq 2$, satisfying a series of coherence conditions, the first of which says that $\omega_2$ generates chain homotopies controlling the curvature of $\nabla$. Intuitively, one could think of a representation up to homotopy as a nice resolution of a representation on the (possibly singular) homology of the chain complex. Representations up to homotopy provide a useful framework for studying deformation theory \cite{aba-sha:deformations, aba-cra:rephom} and constructing characteristic classes \cite{cra-fer:secondary, gra-meh:vbalg} for Lie algebroids.

The primary purpose of this paper is to connect the notion of representation up to homotopy to that of \emph{Lie algebroid module}. The latter was introduced by Vaintrob \cite{vaintrob}, using the language of supergeometry. A module over a Lie algebroid $A$ is defined to be an $NQ$-vector bundle over $A[1]$. To our knowledge, Lie algebroid modules are the first generalized Lie algebroid representations to appear in the literature. More important, since Lie algebroid modules are defined in terms of vector bundles, it is straightforward to define many constructions, such as duals and tensor products, which one would expect a good theory of representations to have; in particular, the adjoint representation is just the tangent bundle.

Our main result is Theorem \ref{thm:correspondence}, which states that, up to isomorphism, there is a one-to-one correspondence between Lie algebroid modules and semibounded representations up to homotopy. This correspondence arises from a process called ``decomposition'' that allows one to obtain a representation up to homotopy from a Lie algebroid module. Although decomposition is noncanonical, different choices lead to representations up to homotopy that are equivalent in a certain way. This fact explains why the adjoint module (i.e. the tangent bundle) is canonical, whereas the adjoint representation up to homotopy is only well-defined up to equivalence.

Theorem \ref{thm:correspondence} extends a result due to Gracia-Saz and the author \cite{gra-meh:vbalg}, where it was shown that a similar correspondence holds between $2$-term representations up to homotopy and $\VB$-algebroids. Thus, one could interpret the result of this paper as asserting that Lie algebroid modules provide the natural extension of the category of $\VB$-algebroids to a category where tensor products exist.

Finally, we remark that most of the results of this paper are consequences of two structure theorems (Theorems \ref{thm:structure} and \ref{thm:structure2}) that are proven in the general setting of vector bundles over $\N$-graded manifolds.  Therefore, the results of this paper could be applied to the representation theory of other structures that have supergeometric descriptions, including Lie $n$-algebroids, $L_\infty$-algebras, and Courant algebroids.

The structure of the paper is as follows:
\begin{itemize}
     \item In \S\ref{sec:structure}, we study vector bundles over $N$-manifolds. We state and prove the structure theorems and introduce the notion of decomposition.
     \item In \S\ref{sec:rephom}, we recall the definitions of representation up to homotopy and gauge-equivalence.
     \item In \S\ref{sec:modules}, we recall the definition of Lie algebroid module, and we arrive at the main results relating Lie algebroid modules to representations up to homotopy.
     \item In \S\ref{sec:adjoint}, we consider the example of the adjoint module of a Lie algebroid $A$. The cohomology of $A$ with values in the adjoint module is isomorphic to the deformation cohomology of Crainic and Moerdijk \cite{cra-moe:deform}.
     \item In \S\ref{sec:tensor}, we describe the constructions of tensor product, direct sum, and dual. We show that there is a cohomology pairing for dual Lie algebroid modules.
     \item In \S\ref{sec:charclasses}, we recall the construction of characteristic classes in \cite{gra-meh:vbalg}, and show that this construction provides well-defined invariants of Lie algebroid modules.
\end{itemize}

\begin{ack}
	We thank David Li-Bland, Dmitry Roytenberg, and Jim Stasheff for helpful comments and suggestions on earlier versions of the paper. We additionally thank the anonymous referee for several suggestions that improved the exposition of the paper.
\end{ack}

\section{The structure of $N$-manifold vector bundles}\label{sec:structure}

Throughout this paper, we will be working in the category of graded manifolds. We refer the reader to \cite{mehta:thesis, meh:qalg, voronov:graded,roytenberg:graded,catt-schatz:super} (although, in contrast to \cite{voronov:graded}, we adopt a definition for which a function's parity agrees with its weight or degree). In particular, a brief introduction to vector bundles in the category of graded manifolds is given in \cite{meh:qalg}.

Let $\cM$ be a nonnegatively graded manifold, or $N$-manifold. Recall that there is a natural projection $\pi_{\cM}$ onto the underlying degree $0$ manifold $M$, where the pullback map $\pi_{\cM}^*$ identifies smooth functions on $M$ with degree $0$ functions on $\cM$. There is also a natural ``zero'' embedding $0_\cM : M \to \cM$, whose pullback map annihilates the ideal of positive degree functions on $\cM$. 

A vector bundle $\cB$ over $\cM$ is given by its sheaf of sections $\Gamma(\cB)$, which is, by definition, a locally free graded $C^\infty(\cM)$-module. We denote the rank of $\Gamma(\cB)$ in degree $i$ by $\rank_i(\cB)$. For simplicity, we suppose that $\cB$ is degree-bounded, in the sense that there exist integers $m, n$ such that $\rank_i(\cB) = 0$ for $i < m$ and for $i > n$. However, as noted below in Remark \ref{rmk:bounded}, the results of this section continue to hold if $\rank_i(\cB)$ is only bounded on one side. In any case, we emphasize that the total space of $\cB$ is allowed to be a $\Z$-graded (as opposed to $\N$-graded) manifold.

The pullback bundle $0_\cM^* \cB$ is a graded vector bundle over $M$. Any graded $C^\infty(M)$-module canonically splits as a direct sum of its homogeneous parts, so we may write $0_\cM^* \cB = \bigoplus E_i[-i]$, where $\{E_i\}$ is a collection of vector bundles over $M$.  We refer to $\cE \defequal 0_\cM^* \cB = \bigoplus E_i[-i]$ as the \emph{standard graded vector bundle} associated to $\cB$. Obviously, $\rank_i(\cB) = \rank(E_i)$.

For each integer $i$, let $\cF_i(\cB)$ denote the sub-$C^\infty(\cM)$-module of $\Gamma(\cB)$ generated by sections of degree $\leq i$. Since we are assuming that the rank of $\cB$ is bounded below by $m$ and above by $n$, we have that $\cF_i(\cB) = 0$ for $i < m$, and that $\cF_n(\cB) = \Gamma(\cB)$. We then have a filtration
\begin{equation}\label{eqn:filtration}
0 \subseteq \cF_m(\cB) \subseteq \cF_{m+1}(\cB) \subseteq \cdots \subseteq \cF_{n}(\cB) = \Gamma(\cB).
\end{equation}
The key observation we wish to make is that the quotient $\cF_i(\cB)/\cF_{i-1}(\cB)$ is naturally isomorphic to $C^\infty(\cM) \otimes_{C^\infty(M)} \Gamma(E_i[-i])$. The latter may be viewed geometrically as the space of sections of $\pi_\cM^* E_i[-i]$, so the sum $\bigoplus \cF_i(\cB)/\cF_{i-1}(\cB)$ is isomorphic to $\Gamma(\pi_\cM^* \cE)$.

By choosing splittings of the short exact sequence
\begin{equation}\label{eqn:sesbk}
\cF_{i-1}(\cB) \to \cF_i(\cB) \to \Gamma(\pi_\cM^*E_i[-i])
\end{equation}
for each $i$, we obtain an isomorphism $\cB \iso \pi_\cM^*  \cE$. Thus we have the following structure theorem:
\begin{thm}\label{thm:structure}
     Let $\cB$ be a vector bundle over $\cM$, and let $\cE \to M$ be the standard graded vector bundle associated to $\cB$. Then $\cB$ is noncanonically isomorphic to $\pi_\cM^*(\cE)$.
\end{thm}

The statement of Theorem \ref{thm:structure} can be strengthened slightly. We have described a specific procedure for constructing isomorphisms from $\cB$ to $\pi_\cM^*(\cE)$, and we would like to characterize the isomorphisms that arise from this procedure, as well as to describe the difference between any two such isomorphisms. To address this issue, we first make the observation that $0_\cM^* \cB = \cE$ is canonically isomorphic to $0_\cM^* \pi_\cM^* \cE$, since $\pi_\cM \circ 0_\cM = \id_M$. The isomorphisms $\Theta: \cB \to \pi_\cM^*(\cE)$ obtained via splittings of \eqref{eqn:sesbk} are those that are filtration-preserving, and such isomorphisms satisfy the property that the following diagram commutes:
\begin{equation}\label{eqn:0bundle}
     \xymatrix{ \cE \ar_{\tilde{0}_\cM}[d] \ar^\id[r] & \cE \ar^{\tilde{0}_\cM}[d] \\ \cB \ar^-\Theta[r] & \pi_\cM^*(\cE)}
\end{equation}
Here, the vertical maps are the natural maps associated to pullback bundles.

On the other hand, by considering changes of splittings of the sequences \eqref{eqn:sesbk}, we see that the difference between any two filtration-preserving isomorphisms is given by a collection of maps $\sigma_{k,i} : \Gamma(E_k) \to C^\infty_{i}(\cM) \otimes \Gamma(E_{k-i})$ for $1 \leq i \leq k-m$. The associated automorphism of $\pi_\cM^*(\cE)$ takes $\varepsilon \in \Gamma(E_k)$ to $\varepsilon + \sum_{i=1}^{k-m} \sigma_{k,i}(\varepsilon)$. All automorphisms of $\pi_\cM^*(\cE)$ fixing the image of $\tilde{0}_\cM$ are of this form. In summary, we have the following result, which refines Theorem \ref{thm:structure}:

\begin{thm}\label{thm:structure2}
     Let $\cB$ be a vector bundle over $\cM$, and let $\cE \to M$ be the standard graded vector bundle associated to $\cB$. An isomorphism $\Theta: \cB \to \pi_\cM^*(\cE)$ is filtration-preserving if and only if the diagram \eqref{eqn:0bundle} commutes.
\end{thm}

For later use, we introduce the following terminology.
\begin{definition}\label{dfn:decomp}
     A \emph{decomposition} of a vector bundle $\cB \to \cM$ is a choice of isomorphism $\Theta: \cB \to \pi_\cM^*(\cE)$ such that \eqref{eqn:0bundle} commutes.
\end{definition}

\begin{definition}\label{dfn:stat}
     A \emph{statomorphism} of a vector bundle $\cB \to \cM$ is a vector bundle automorphism $\Psi$ such that
\begin{equation}
     \xymatrix{ \cE \ar_{\tilde{0}_\cM}[d] \ar^\id[r] & \cE \ar^{\tilde{0}_\cM}[d] \\ \cB \ar^-\Psi[r] & \cB}
\end{equation}
commutes.
\end{definition}

Using the terminology of Definitions \ref{dfn:decomp} and \ref{dfn:stat}, we can restate the above structure theorems as follows. Theorem \ref{thm:structure} gives the existence of decompositions, and Theorem \ref{thm:structure2} says that the statomorphisms (which form a group) act freely and transitively on the space of decompositions. As noted above, a statomorphism is given by a collection of maps $\sigma_{k,i} : \Gamma(E_k) \to C^\infty_{i}(\cM) \otimes \Gamma(E_{k-i})$ for $1 \leq i \leq k-m$.

\begin{remark}
	The term ``statomorphism'' is due to Gracia-Saz and Mackenzie \cite{gra-mac}, who used it to describe automorphisms of double and triple vector bundles that preserve the underlying structure bundles. We use the term here because there is a natural way to view double vector bundles as graded vector bundles (for example, see \cite{gra-rot,meh:qalg,royt:thesis}), and in this case our definition of statomorphism coincides with theirs.
\end{remark}

\begin{remark}\label{rmk:bounded}
The structure theorems in this section can be extended to the case where $\rank_i(\cB)$ is only bounded one side. In this ``semibounded'' case, the filtration in \eqref{eqn:filtration} would extend infinitely in one direction. By choosing splittings of the short exact sequences \eqref{eqn:sesbk}, we can obtain an isomorphism $\cB \iso \pi_\cM^* \cE$ as a colimit of isomorphisms.
\end{remark}

\section{Representations up to homotopy of Lie algebroids}\label{sec:rephom}

Let $A \to M$ be a Lie algebroid. Then $\Omega(A) \defequal \bigwedge \Gamma(A^*)$ is the algebra of \emph{$A$-forms}, equipped with the differential $d_A$.

Let $\cE = \bigoplus E_i[-i]$ be a graded vector bundle over $M$. The space of \emph{$\cE$-valued $A$-forms}
\[ \Omega(A;\cE) \defequal \Omega(A) \otimes_{C^\infty(M)} \Gamma(\cE)\]
is endowed with a $\Z$-grading where the subspace $\Omega^p(A) \otimes \Gamma(E_i[-i])$ is homogeneous of degree $p+i$. 

\begin{definition}\label{dfn:rephom}
     A \emph{representation up to homotopy}, or \emph{$\infty$-representation}, of $A$ on $\cE$ is a degree $1$ operator $\totaldiff$ on $\Omega(A;\cE)$ such that $\totaldiff^2=0$ and such that the Leibniz rule
\begin{equation}\label{eqn:leibniz}
   \totaldiff(\alpha \omega) = (d_A \alpha) \omega + (-1)^{p} \alpha (\totaldiff \omega)    
\end{equation}
holds for $\alpha \in \Omega^p(A)$ and $\omega \in \Omega(A;\cE)$.
\end{definition}

There is a natural projection map $\mu: \Omega(A;\cE) \to \Gamma(\cE)$ for which the kernel is $\bigoplus_{p>0} \Omega^p(A) \otimes \Gamma(\cE)$. If $\totaldiff$ is an $\infty$-representation, then the Leibniz rule implies that $\ker \mu$ is $\totaldiff$-invariant. Therefore, there is an induced differential $\boundary$ on $\Gamma(\cE)$, defined by the property that the following diagram commutes:
\[ \xymatrix{ \Omega(A;\cE) \ar^\totaldiff[r] \ar_\mu[d] &  \Omega(A;\cE) \ar_\mu[d] \\ \Gamma(\cE) \ar^\boundary[r] & \Gamma(\cE)}\]

If $\totaldiff$ and $\totaldiff'$ are $\infty$-representations of $A$ on $\cE$ and $\cE'$, respectively, then a \emph{morphism} from $\totaldiff$ to $\totaldiff'$ is an $\Omega(A)$-module morphism $\phi: \Omega(A;\cE) \to \Omega(A;\cE')$ such that $\phi \circ \totaldiff = \totaldiff' \circ \phi$. In this case, $\phi$ induces a chain map from $\Gamma(\cE)$ to $\Gamma(\cE')$.

As usual, an invertible morphism of $\infty$-representations is called an isomorphism. However, in the case where the graded vector bundle $\cE$ is fixed, there is a slightly more refined notion, which we call \emph{gauge equivalence}.

\begin{definition}
     A \emph{gauge transformation} of $\Omega(A;\cE)$ is a degree-preserving $\Omega(A)$-module automorphism $u$ such that the following diagram commutes:
\[ \xymatrix{ \Omega(A;\cE) \ar^u[r] \ar_\mu[d] &  \Omega(A;\cE) \ar_\mu[d] \\ \Gamma(\cE) \ar^\id[r] & \Gamma(\cE)}\]
\end{definition}

Under a gauge transformation, an $\infty$-representation $\totaldiff$ transforms as $\totaldiff' = u^{-1} \totaldiff u$. Two $\infty$-representations that are related by a gauge transformation are said to be \emph{gauge-equivalent}. Note that gauge-equivalent $\infty$-representations induce the same differential $\boundary$ on $\Gamma(\cE)$.

\section{Lie algebroid modules}\label{sec:modules}

Let $A \to M$ be a Lie algebroid.

\begin{definition}[\cite{vaintrob}] \label{dfn:module}
     A \emph{Lie algebroid module} over $A$, or \emph{$A$-module}, is a vector bundle $\cB \to A[1]$ equipped with a degree $1$ operator $Q$ on $\Gamma(\cB)$ such that $Q^2=0$ and such that the Leibniz rule
\[ Q(\alpha \beta) = (d_A \alpha) \beta + (-1)^{p} \alpha (Q \beta) \]
holds for $\alpha \in C^\infty_p(A[1]) = \Omega^p(A)$ and $\beta \in \Gamma(\cB)$. 
\end{definition}
A \emph{morphism} of $A$-modules from $(\cB, Q)$ to $(\cB', Q')$ is a linear map $\psi: \cB \to \cB'$, covering the identity map on $A[1]$, such that $\psi Q = Q' \psi$.

Recall that, for a fixed vector bundle $\cB \to A[1]$, we have defined in Definition \ref{dfn:stat} a distinguished class of automorphisms, called \emph{statomorphisms}. We will say that two $A$-module structures $Q$ and $Q'$ on $\cB$ are \emph{statomorphic} if there exists a statomorphism $\psi: \cB \to \cB$ such that $\psi Q = Q' \psi$.

\begin{remark}
The operator $Q$ in the definition of Lie algebroid module can be equivalently viewed as a linear homological vector field whose base vector field is $d_A$. In other words, a Lie algebroid module is a special case of a \emph{$Q$-vector bundle}, i.e.\ a vector bundle in the category of $Q$-manifolds.
\end{remark}

\begin{remark}
	Of particular interest is the special case where the total space of $\cB$ is concentrated in degrees $0$ and $1$ (so that $\rank_i(\cB)$ vanishes except for $i=-1,0$). In this case, $\cB = D[1]$ for some vector bundle $D \to E$. The fact that $\cB$ also has a vector bundle structure over $A[1]$ implies that $D$ is a double vector bundle. In this case, an $A$-module structure on $\cB$ is equivalent to a $\VB$-algebroid structure on $D$ over $A$ (see \cite{gra-meh:vbalg}).
\end{remark}

Let $\cE = \bigoplus E_i[-i]$ be a graded vector bundle over $M$.  Assume that $\cE$ is bounded in degree (or semibounded, c.f.\ Remark \ref{rmk:bounded}). Initially, we consider $A$-module structures of the form $\pi_A^* \cE \to A[1]$, where $\pi_A$ is the projection map from $A[1]$ to $M$. In this case, the module of sections $\Gamma(\pi_A^*\cE)$ is canonically isomorphic to $\Omega(A) \otimes_{C^\infty(M)} \Gamma(\cE) = \Omega(A;\cE)$. Under this identification, Definitions \ref{dfn:rephom} and \ref{dfn:module} become identical, so we immediately have the following:
\begin{lemma}\label{lemma:correspondence}
$\infty$-representations of $A$ on $\cE$ are in one-to-one correspondence with $A$-modules of the form $\pi_A^* \cE$.
\end{lemma}

In light of Theorem \ref{thm:structure2}, we have a straightforward way to obtain an $\infty$-rep\-re\-sen\-ta\-tion from an arbitrary $A$-module $\cB \to A[1]$; one simply needs to choose a decomposition $\Theta: \cB \to \pi^*_A \cE$ (see Definition \ref{dfn:decomp}), and then the operator $Q$ on $\Gamma(\cB)$ induces an $\infty$-representation $\totaldiff \defequal \Theta \circ Q \circ \Theta^{-1}$ of $A$ on $\cE$. 

Furthermore, we observe that a gauge transformation of $\Omega(A;\cE)$ is precisely the same thing as an automorphism of $\pi_A^*(\cE)$ that preserves the image of $\tilde{0}_A$. In other words, changes of decomposition correspond to gauge transformations of $\totaldiff$. We now have our main result:

\begin{thm}\label{thm:correspondence}
Let $A \to M$ be a Lie algebroid.
\begin{enumerate}
     \item There is a one-to-one correspondence between isomorphism classes of $A$-modules and isomorphism classes of (semi)bounded $\infty$-representations of $A$.
     \item For any (semi)bounded graded vector bundle $\cE = \bigoplus E_i \to M$, there is a one-to-one correspondence between statomorphism classes of $A$-modules with standard graded vector bundle $\cE$ and gauge-equivalence classes of $\infty$-representations of $A$ on $\cE$.
\end{enumerate}
\end{thm}

\section{Adjoint module and deformation cohomology}\label{sec:adjoint}

Let $A \to M$ be a Lie algebroid, and let $\cB \to A[1]$ be an $A$-module.
\begin{definition}[\cite{vaintrob}]
     The \emph{cohomology of $A$ with values in $\cB$}, denoted $H^\bullet(A;\cB)$, is the cohomology of the complex $(\Gamma(\cB), Q)$.
\end{definition}
The results of \S\ref{sec:modules} imply that $H^\bullet(A;\cB)$ is isomorphic to the cohomology of $A$ with values in any $\infty$-representation arising from $\cB$. 

\begin{example}[Adjoint module]
The \emph{adjoint module} of $A$ is the tangent bundle $T(A[1])$. The sections of $T(A[1])$ are vector fields on $A[1]$ (i.e., graded derivations of the algebra $\Omega(A)$), with the operator $Q \defequal [d_A, \cdot]$. The low-degree cohomology with values in the adjoint module was briefly described in \cite{vaintrob}, but we provide additional details here.

We first consider degree $-1$. The degree $-1$ derivations of $\Omega(A)$ are precisely the contraction operators $\iota_X$ for $X \in \Gamma(A)$. The Lie derivative operator $L_X \defequal [d_A, \iota_X]$ vanishes if and only if $X$ is in the center of the Lie algebra $\Gamma(A)$. Therefore $H^{-1}(A; TA[1])$ can be identified with the center of $\Gamma(A)$.

Next, we consider degree $0$. The degree $0$ derivations of $\Omega(A)$ are in one-to-one correspondence with linear vector fields on $A$. A degree $0$ derivation $\phi$ satisfies the equation $[d_A, \phi] = 0$ if and only if $\phi$ corresponds to a \emph{morphic} vector field \cite{mac-xu:mult, meh:qalg}, i.e.\ an infinitesimal automorphism of $A$. The coboundaries are the Lie derivatives $L_X$, which may be considered inner infinitesimal automorphisms. The cohomology $H^0(A;TA[1])$ is then the space of outer infinitesimal automorphisms.

Let $\chi$ be a degree $1$ derivation, and consider the operator $d_A + \chi h$, where $h$ is a formal parameter. Then $(d_A + \chi h)^2$ vanishes to order $h^2$ if and only if $[d_A, \chi] = 0$. Thus, the degree $1$ cocycles correspond to infinitesimal deformations of the Lie algebroid structure on $A$. The coboundaries consist of those ``trivial'' infinitesimal deformations that come from pulling back $d_A$ along infinitesimal bundle automorphisms of $A$. In this sense, $H^1(A;TA[1])$ controls the infinitesimal deformations of $A$.

The degree $2$ cohomology arises when one wants to extend an infinitesimal deformation to higher order. For example, suppose that $\chi$, as above, is a degree $1$ cocycle. Then $\chi^2$ is a degree $2$ cocycle. If $\chi^2 = -[d_A, \nu]$, then $(d_A + \chi h + \nu h^2)^2$ vanishes to order $h^3$. More generally, given a formal operator $d_A + \sum_{i=1}^k \chi_i h^i$ whose square vanishes to order $h^k$, one can find a $\chi_{k+1}$ such that $(d_A + \sum_{i=1}^{k+1} \chi_i h^i)^2$ vanishes to order $h^{k+1}$ if an obstruction in $H^2(A;TA[1])$, depending on the $\chi_i$, vanishes.
\end{example}

It was observed by Crainic and Moerdijk \cite{cra-moe:deform} that the differential graded Lie algebra of derivations of $\Omega(A)$ is isomorphic, up to a degree shift, with their \emph{deformation complex} of $A$, consisting of $k$-ary antisymmetric brackets on $\Gamma(A)$ satisfying Leibniz rules. The isomorphism can be described in terms of derived brackets, as follows. Let $\chi$ be a degree $k$ derivation of $\Omega(A)$. Then we may define a $(k+1)$-ary bracket $\dbracket{\cdot,\dots,\cdot}_\chi$ on $\Gamma(A)$ by 
\[ \iota_{\dbracket{X_1,\dots,X_{k+1}}_\chi} = [[\cdots [[\chi, \iota_{X_1}], \iota_{X_2}], \cdots], \iota_{X_{k+1}}]\]
for $X_1, \dots, X_{k+1} \in \Gamma(A)$. Antisymmetry of $\dbracket{\cdot,\dots,\cdot}_\chi$ follows from the Jacobi identity and the fact that contraction operators commute. The Leibniz rule follows immediately from the fact that the Lie bracket of derivations satisfies the Leibniz rule.

\section{Tensor products, direct sums, and duals}\label{sec:tensor}

Let $A \to M$ be a Lie algebroid, and let $(\cB_1, Q_1)$ and $(\cB_2, Q_2)$ be $A$-modules. Then there is a natural $A$-module structure on $\cB_1 \otimes \cB_2$, given by
\[ Q(\beta_1 \otimes \beta_2) = (Q_1 \beta_1) \otimes \beta_2 + (-1)^{|\beta_1|} \beta_1 \otimes (Q_2 \beta_2)\]
for $\beta_i \in \Gamma(\cB_i)$. The tensor product is symmetric, in the sense that the Koszul isomorphism from $\cB_1 \otimes \cB_2$ to $\cB_2 \otimes \cB_1$, taking $\beta_1 \otimes \beta_2$ to $(-1)^{|\beta_1||\beta_2|} \beta_2 \otimes \beta_1$, is an $A$-module isomorphism.

Similarly, the direct sum $\cB_1 \oplus \cB_2$ inherits an $A$-module structure, given by
\[ Q(\beta_1 + \beta_2) = Q_1 \beta_1 + Q_2 \beta_2.\]

Next, we consider duals. Let $(\cB,Q)$ be an $A$-module, and let $\cB^* \to A[1]$ be the vector bundle dual to $\cB$. We denote by $\innerprod{\cdot,\cdot}$ the pairing taking $\Gamma(\cB^*) \otimes \Gamma(\cB)$ to $C^\infty(A[1]) = \Omega(A)$. The induced $A$-module structure $Q^*$ on $\cB^*$ is uniquely determined by the equation
\begin{equation}\label{eqn:dual}
 d_A \innerprod{b,\beta} = \innerprod{Q^*b, \beta} + (-1)^{|b|} \innerprod{b,Q \beta}     
\end{equation}
for $b \in \Gamma(\cB^*)$ and $\beta \in \Gamma(\cB)$. Note that requiring \eqref{eqn:dual} to hold is equivalent to asking that the pairing $\innerprod{\cdot,\cdot}$ be an $A$-module morphism from $\cB^* \otimes \cB$ to the trivial rank $1$ $A$-module $(A[1] \times \R, d_A)$.

Dualization takes vector bundles that are bounded in degree from below to those that are bounded from above, and vice versa. The property of being bounded on both sides is preserved by dualization.

\begin{proposition}\label{prop:pairing}
     The pairing between $\Gamma(\cB)$ and $\Gamma(\cB^*)$ induces a well-defined cohomology pairing $H^\bullet(A;\cB^*) \otimes H^\bullet(A;\cB) \to H^\bullet(A)$.
\end{proposition}
\begin{proof}
From \eqref{eqn:dual}, we have that $\innerprod{b,\beta}$ is closed if both $b$ and $\beta$ are closed, and that $\innerprod{b,\beta}$ is exact if one of $b$ or $\beta$ is exact and the other is closed. Therefore, the map taking $[b] \otimes [\beta]$ to $[\innerprod{b,\beta}]$ is well-defined at the level of cohomology.
\end{proof}

In the case where $M$ is compact and orientable, one can obtain $\R$-valued pairings parametrized by cohomology with values in the (canonically decomposed) Berezinian $A$-module $\Ber \defequal \pi_A^*(\wedge^\top A \otimes \wedge^\top T^*M)$. This is done by composing the pairing of Proposition \ref{prop:pairing} with that of Evens, Lu, and Weinstein \cite{elw}.

\section{Characteristic classes}\label{sec:charclasses}
Chern-Simons type characteristic classes associated to $\infty$-representations were constructed in \cite{gra-meh:vbalg}.  In the cases of the adjoint $\infty$-representation and of genuine representations, these classes coincide with those constructed by Crainic and Fernandes \cite{fer:characteristic, cra:vanest, cra-fer:secondary}. In particular, the degree $1$ characteristic class agrees with the modular class \cite{elw}.

In the $2$-term case, it was shown there that the characteristic classes are gauge-invariant, so they can be interpreted as $\VB$-algebroid invariants. We recall the construction here, and we show in Theorem \ref{thm:gauge} that the gauge-invariance property holds in full generality.

Let $A \to M$ be a Lie algebroid, and let $\cE \to M$ be a graded vector bundle that is bounded in degree. We recall the notion of \emph{$A$-superconnection}.
\begin{definition}[\cite{gra-meh:vbalg}]
An \emph{$A$-superconnection} on $\cE$ is a degree $1$ operator $\totaldiff$ on $\Omega(A;\cE)$ satisfying the Leibniz rule \eqref{eqn:leibniz}. An $A$-superconnection is called \emph{flat} if $\totaldiff^2=0$.
\end{definition}
Clearly, a flat $A$-superconnection is the same thing as an $\infty$-representation. In general, a version of Chern-Weil theory gives obstructions to the existence of $\infty$-representations. Specifically, one can choose any $A$-superconnection $\totaldiff$ and obtain the Chern-Weil forms
\[ \ch_k(\totaldiff) \defequal \str(\totaldiff^{2k}) \in \Omega^{2k}(A),\]
where $\str$ denotes the supertrace. These are straightforward generalizations of the forms considered by Quillen \cite{quillen:superconnections}, and his proof of the following statement carries over almost verbatim to the present setting.
\begin{proposition}\label{prop:ch}
For each $k$, the form $\ch_k(\totaldiff)$ is closed, and the cohomology class of $\ch_k(\totaldiff)$ is independent of $\totaldiff$.
\end{proposition}

In the case where $\totaldiff$ is an $\infty$-representation, the Chern-Weil forms $\ch_k(\totaldiff)$ obviously vanish. However, given a pair of $\infty$-representations, one can construct Chern-Simons type transgression forms, as follows.

Let $I$ be the unit interval, and consider the product Lie algebroid $A \times TI \to M \times I$. Let $\{t,\dot{t}\}$ be the canonical coordinates on $T[1]I$. Any Lie algebroid $q$-form $\xi \in \Omega^q(A \times TI)$ can be uniquely written as $\xi_0(t) + \dot{t} \xi_1(t)$, where $\xi_0$ and $\xi_1$ are $t$-dependent elements of $\Omega^q(A)$ and $\Omega^{q-1}(A)$, respectively. The Berezin integral
\begin{equation*}
     \int \xi \defequal \int_{T[1]I} \dee{t} \dee{\dot{t}} \xi = \int_0^1 \dee{t} \xi_1
\end{equation*}
defines a degree $-1$ map from $\Omega(A \times TI)$ to $\Omega(A)$. The differential on $\Omega(A\times TI)$ is
\[ d_{A \times TI} = d_A + \dot{t} \ddt,\]
and a straightforward computation shows that the equation
\begin{equation}\label{eqn:dint}
     \int d_{A \times TI} \xi + d_A \int \xi = \xi_0(1) - \xi_0(0)
\end{equation}
holds for all $\xi \in \Omega(A \times TI)$.

Let $p$ be the projection map from $M \times I$ to $M$. Given a pair of $A$-superconnections $\totaldiff_0$ and $\totaldiff_1$ on $\cE$, we can form an $(A \times TI)$-superconnection $\trans{\totaldiff_0}{\totaldiff_1}$ on $p^* \cE$, given by
\[ \trans{\totaldiff_0}{\totaldiff_1}(a) = t \totaldiff_1(a) + (1-t) \totaldiff_0(a), \]
where $a \in \Gamma(\cE)$ is viewed as a $t$-independent section of $p^* \cE$. The \emph{transgression forms} $\cs_k(\totaldiff_0,\totaldiff_1) \in \Omega^{2k-1}(A)$ are defined as
\[ \cs_k(\totaldiff_0,\totaldiff_1) \defequal \int \ch_k(\trans{\totaldiff_0}{\totaldiff_1}) = \int \str(\trans{\totaldiff_0}{\totaldiff_1}^{2k}).\]
\begin{proposition}\label{prop:dacs}
$d_A \cs_k(\totaldiff_0,\totaldiff_1) = \ch_k(\totaldiff_1) - \ch_k(\totaldiff_0)$. In particular, if $\totaldiff_0$ and $\totaldiff_1$ are $\infty$-representations, then $\cs_k(\totaldiff_0,\totaldiff_1)$ is closed.
\end{proposition}
\begin{proof}
     Let $\xi \defequal \ch_k(\trans{\totaldiff_0}{\totaldiff_1}) = \str(\trans{\totaldiff_0}{\totaldiff_1}^{2k}) \in \Omega^{2k}(A \times TI)$. By Proposition \ref{prop:ch}, we have that $d_{A \times TI} \xi = 0$. Equation \eqref{eqn:dint} then implies that 
\begin{equation}\label{eqn:dacs}
d_A \cs_k(\totaldiff_0,\totaldiff_1) = \xi_0(1) - \xi_0(0).     
\end{equation}
To compute the right side of \eqref{eqn:dacs}, we first calculate
\begin{equation}\label{eqn:t2}
  \trans{\totaldiff_0}{\totaldiff_1}^2 = \dot{t}(\totaldiff_1 - \totaldiff_0) + t^2 \totaldiff_1^2 + (1-t)^2 \totaldiff_0^2 + t(1-t)[\totaldiff_0,\totaldiff_1],    
\end{equation}
so that
\begin{align*}
     \trans{\totaldiff_0}{\totaldiff_1}^2(1) &= \dot{t}(\totaldiff_1 - \totaldiff_0) + \totaldiff_1^2,\\
\trans{\totaldiff_0}{\totaldiff_1}^2(0) &= \dot{t}(\totaldiff_1 - \totaldiff_0) + \totaldiff_0^2.
\end{align*}
It follows that
\begin{align*}
     \trans{\totaldiff_0}{\totaldiff_1}^{2k}(1) &= \totaldiff_1^{2k} + \bigO(\dot{t}),\\
\trans{\totaldiff_0}{\totaldiff_1}^{2k}(0) &= \totaldiff_0^{2k} + \bigO(\dot{t}).
\end{align*}
We conclude that the right side of \eqref{eqn:dacs} is $\str(\totaldiff_1^{2k}) - \str(\totaldiff_0^{2k}) = \ch_k(\totaldiff_1) - \ch_k(\totaldiff_0)$.
\end{proof}

\begin{remark}
	If $\totaldiff_0$ and $\totaldiff_1$ are $\infty$-representations, then \eqref{eqn:t2} reduces to
\begin{equation*}
	  \trans{\totaldiff_0}{\totaldiff_1}^2 = \dot{t}(\totaldiff_1 - \totaldiff_0) + t(1-t)[\totaldiff_0,\totaldiff_1].
\end{equation*}
Using the fact that $\totaldiff_0$ and $\totaldiff_1$ both commute with $[\totaldiff_0, \totaldiff_1]$, we see that
\begin{equation*}
	  \trans{\totaldiff_0}{\totaldiff_1}^{2k} = k\dot{t}t^{k-1}(t-1)^{k-1}(\totaldiff_1 - \totaldiff_0)[\totaldiff_0,\totaldiff_1]^{k-1} + t^k(1-t)^k[\totaldiff_0,\totaldiff_1]^k.	
\end{equation*}
The Berezin integral of $\trans{\totaldiff_0}{\totaldiff_1}^{2k}$ can then be explicitly computed, giving us the simple formula
\begin{equation}
\cs_k(\totaldiff_0,\totaldiff_1) = P_k \str\left( (\totaldiff_1 - \totaldiff_0) [\totaldiff_0, \totaldiff_1]^{k-1} \right),
\end{equation}
where the constant $P_k$ is
\begin{equation*}
	P_k = k\int_0^1 t^{k-1}(1-t)^{k-1} \dee t = \frac{k!(k-1)!}{(2k-1)!}.
\end{equation*}
\end{remark}

The following two propositions describe important properties satisfied by the forms $\cs_k(\totaldiff_0, \totaldiff_1)$. The first is a sort of ``triangle identity'', and the second asserts that the cohomology classes are stable under $\Omega(A)$-module automorphisms.
\begin{proposition}\label{prop:triangle}
     Let $\totaldiff_0$, $\totaldiff_1$, and $\totaldiff_2$ be $\infty$-representations of $A$ on $\cE$. Then
\[ \cs_k(\totaldiff_0, \totaldiff_1) + \cs_k(\totaldiff_1, \totaldiff_2) - \cs_k(\totaldiff_0, \totaldiff_2)\]
is exact.
\end{proposition}
\begin{proof}
     Consider the transgression form $\xi \defequal \cs_k(\trans{\totaldiff_0}{\totaldiff_1}, \trans{\totaldiff_0}{\totaldiff_2}) \in \Omega^{2k-1}(A \times TI)$. By Proposition \ref{prop:dacs}, we have that
\[ d_{A \times TI} \xi = \ch_k(\trans{\totaldiff_0}{\totaldiff_2}) - \ch_k(\trans{\totaldiff_0}{\totaldiff_1}).\]
Applying the Berezin integral to both sides and using \eqref{eqn:dint}, we get
\[ -d_A \int \xi +  \xi_0(1) - \xi_0(0) = \cs_k(\totaldiff_0,\totaldiff_2) - \cs_k(\totaldiff_0,\totaldiff_1).\]
To complete the proof, we need to compute the terms $\xi_0(1)$ and $\xi_0(0)$.

Letting $s$ be the coordinate on the second copy of $I$, we write
\[ \trans{\trans{\totaldiff_0}{\totaldiff_1}}{\trans{\totaldiff_0}{\totaldiff_2}}(a) = st \totaldiff_2(a) + (1-s)t \totaldiff_1(a) + (1-t)\totaldiff_0(a)\]
for $a$ an $s$- and $t$-independent section of the pullback of $\cE$ to $M \times I \times I$. Then
\begin{equation*}
     \begin{split}
\trans{\trans{\totaldiff_0}{\totaldiff_1}}{\trans{\totaldiff_0}{\totaldiff_2}}^2 =& \dot{s} t(\totaldiff_2 - \totaldiff_1) + \dot{t}(s \totaldiff_2 + (1-s)\totaldiff_1 - \totaldiff_0) + s(1-s)t^2[\totaldiff_2,\totaldiff_1]          \\
 &+ st(1-t)[\totaldiff_2,\totaldiff_0] + (1-s)t(1-t)[\totaldiff_1,\totaldiff_0].
     \end{split}
\end{equation*}
We observe that integration with respect to $s$ and $\dot{s}$ commutes with evaluation of $t$ and $\dot{t}$, so we may evaluate first. We see that
\begin{align*}
  \trans{\trans{\totaldiff_0}{\totaldiff_1}}{\trans{\totaldiff_0}{\totaldiff_2}}^2|_{t=1} &= \dot{s} (\totaldiff_2 - \totaldiff_1) + \dot{t}(s \totaldiff_2 + (1-s)\totaldiff_1 - \totaldiff_0) + s(1-s)[\totaldiff_2,\totaldiff_1], \\    
\trans{\trans{\totaldiff_0}{\totaldiff_1}}{\trans{\totaldiff_0}{\totaldiff_2}}^2|_{t=0} &= \dot{t}(s \totaldiff_2 + (1-s)\totaldiff_1 - \totaldiff_0),
\end{align*}
so that
\begin{align*}
  \trans{\trans{\totaldiff_0}{\totaldiff_1}}{\trans{\totaldiff_0}{\totaldiff_2}}^{2k}|_{t=1} &= (\dot{s} (\totaldiff_2 - \totaldiff_1) + s(1-s)[\totaldiff_2,\totaldiff_1])^k + \bigO(\dot{t}), \\    
\trans{\trans{\totaldiff_0}{\totaldiff_1}}{\trans{\totaldiff_0}{\totaldiff_2}}^{2k}|_{t=0} &=  \bigO(\dot{t}).
\end{align*}
Therefore,
\[ \xi_0(1) = \int \str (\dot{s} (\totaldiff_2 - \totaldiff_1) + s(1-s)[\totaldiff_2,\totaldiff_1])^k = \int \str (\trans{\totaldiff_1}{\totaldiff_2}^{2k}) = \cs_k(\totaldiff_1,\totaldiff_2)\]
and $\xi_0(0) = 0$.
\end{proof}

\begin{proposition}\label{prop:auto}
Let $u_r$ be a smooth path of degree-preserving $\Omega(A)$-module automorphisms of $\Omega(A;\cE)$ such that $u_0 = \id$.  Let $\totaldiff$ be an $\infty$-representation of $A$ on $\cE$, and let $\totaldiff_r \defequal u_r^{-1} \totaldiff_0 u_r$.  Then $\cs_k(\totaldiff_0,\totaldiff_1)$ is exact.
\end{proposition}
\begin{proof}
     Write $\totaldiff_r = \totaldiff_0 + \theta_r$, where $\theta_r$ is a path of $\End(\cE)$-valued $A$-forms. Since $\totaldiff_r^2=0$, we have that 
\begin{equation}\label{eqn:thetar}
[\totaldiff_0, \totaldiff_r] = [\totaldiff_0, \theta_r] = - \theta_r^2.     
\end{equation}
From \eqref{eqn:t2} we have
\[   \trans{\totaldiff_0}{\totaldiff_r}^2 = \dot{t}\theta_r + t(1-t)[\totaldiff_0,\totaldiff_r] = \dot{t}\theta_r - t(1-t)\theta_r^2,\]    
so
\[   \trans{\totaldiff_0}{\totaldiff_r}^{2k} = k(t^2-t)^{k-1} \dot{t}\theta_r^{2k-1} + (t^2-t)^k\theta_r^{2k}.\]    
It follows that $\cs_k(\totaldiff_0, \totaldiff_r)$ is proportional to $\str(\theta_r^{2k-1})$, so that $\deed{}{r} \cs_k(\totaldiff_0, \totaldiff_r)$ is proportional to
\begin{equation}\label{eqn:ddrcs}
 (2k-1) \str\left(\deed{\theta_r}{r} \theta_r^{2k-2}\right).     
\end{equation}
We wish to show that \eqref{eqn:ddrcs} is exact. First, we compute that
\begin{equation*}
     \begin{split}
          \deed{\theta_r}{r} = \deed{\totaldiff_r}{r} &= \deed{u_r^{-1}}{r} \totaldiff_0 u_r + u_r^{-1} \totaldiff_0 \deed{u_r}{r} \\
          &= \deed{u_r^{-1}}{r} u_r \totaldiff_r + \totaldiff_r u_r^{-1} \deed{u_r}{r} \\
          &= \left[\totaldiff_r, u_r^{-1} \deed{u_r}{r}\right].
     \end{split}
\end{equation*}
Using the property $[\totaldiff_r, \theta_r^2] = 0$, which follows from \eqref{eqn:thetar}, we deduce that
\[ \deed{\theta_r}{r} \theta_r^{2k-2} = \left[ \totaldiff_r, u_r^{-1} \deed{u_r}{r} \theta^{2k-2} \right].\]
Therefore, we have that \eqref{eqn:ddrcs} equals
\[ (2k-1) \str\left( \left[\totaldiff_r, u_r^{-1} \deed{u_r}{r} \theta^{2k-2}\right] \right) = (2k-1) d_A \str\left( u_r^{-1} \deed{u_r}{r} \theta^{2k-2}\right), \]
which is exact, as desired.
\end{proof}

The transgression form construction allows us to define characteristic classes associated to a single $\infty$-representation $\totaldiff$, as follows. Choose a metric on $E_i$ for each $i$. We may use the metric to obtain an adjoint operator $\totaldiff^\adjoint$ on $\Omega(A;\cE)$, given by the equation
\[ d_A \innerprod{\omega_1, \omega_2} = \innerprod{\totaldiff \omega_1, \omega_2} + (-1)^{|\omega_1|}\innerprod{\omega_1, \totaldiff^\adjoint \omega_2}.\]
 The operator $\totaldiff^\adjoint$ satisfies the Leibniz rule and squares to zero, but it is generally not homogeneous of degree $1$; we say that it is a \emph{nonhomogeneous $\infty$-representation} or \emph{nonhomogeneous flat $A$-superconnection}. We observe that the definitions and proofs from earlier in this section carry over verbatim to the nonhomogeneous case, with the only difference being that the Chern-Weil and Chern-Simons forms may be nonhomogeneous.

Given an $\infty$-representation $\totaldiff$, define the Chern-Simons forms associated to $\totaldiff$ as
\[ \cs_k(\totaldiff) \defequal \cs_k(\totaldiff, \totaldiff^\adjoint).\]
The following theorem implies that the cohomology classes of $\cs_k(\totaldiff)$ are well-defined invariants of Lie algebroid modules.
\begin{thm}\label{thm:gauge}
     The cohomology classes $[\cs_k(\totaldiff)]$ are independent of the choice of metric and invariant with respect to gauge transformations. 
\end{thm}
\begin{proof}
The space of metrics is convex, hence path-connected. Given a path of metrics $\innerprod{\cdot,\cdot}_r$, let $u_r \in \End(\cE)$ be given by
\[ \innerprod{a, a'}_r = \innerprod{u_r(a),a'}_0.\]
Then the corresponding adjoint operators satisfy the equation
\[ \totaldiff^{\adjoint_r} = u_r^{-1} \totaldiff^{\adjoint_0} u_r.\]
Metric-independence then follows directly from Propositions \ref{prop:triangle} and \ref{prop:auto}.

Similarly, gauge-invariance follows from Propositions \ref{prop:triangle} and \ref{prop:auto}. For this, we use the fact that the space of gauge transformations is $\bigoplus_{i,k} \Omega^i(A) \otimes \Hom(E_k, E_{k-i})$, which is path-connected.
\end{proof}

\bibliographystyle{amsalpha}
\bibliography{QVB-bib}

\end{document}